\documentclass[11pt, a4]{article}

\usepackage{amssymb}
\usepackage{amsmath}
\usepackage{amsthm}
\usepackage{enumerate}
\usepackage{a4}

\theoremstyle{plain}
\newtheorem{thm}{Theorem}[section]
\newtheorem{lemma}[thm]{Lemma}
\newtheorem{prop}[thm]{Proposition}
\newtheorem{cor}[thm]{Corollary}

\theoremstyle{definition}
\newtheorem{defn}[thm]{Definition}

\newtheoremstyle{theoremnumstyle}
    {\topsep}{\topsep}              
    {\itshape}                      
    {}                              
    {\bfseries}                     
    {.}                             
    { }                             
    {\thmname{#1}\thmnote{ \bfseries #3}}
\theoremstyle{theoremnumstyle}

\usepackage{xcolor}
\usepackage{hyperref}
\hypersetup{
    colorlinks,
    allcolors=blue
}

\newcommand{\from}	{\ensuremath{\colon}}
\renewcommand{\to}	{\ensuremath{\rightarrow}}

\newcommand{\Sn}{{\ensuremath{\mathrm{S}_n}}}
\newcommand{\Sym}{{\ensuremath{\mathrm{S}}}}
\newcommand{\Alt}{{\ensuremath{\mathrm{A}}}}

\newcommand{\myhat}[1]    {\ensuremath{\widehat{#1}}}
\let\oldGamma\Gamma
\renewcommand{\Gamma}{{\hspace{0.5pt} \oldGamma}}

\DeclareMathOperator{\sym}{Sym}

\DeclareMathOperator{\soc}{soc}
\DeclareMathOperator{\aut}{Aut}
\DeclareMathOperator{\out}{Out}

\DeclareMathOperator{\Norm}{{\mathrm{N}}}

\newcommand{\intersection}{\mathbin{\cap}}

\newcommand{\abs}[1]	{%
	\ensuremath{
		\left| #1 \right|
	}%
}
\newcommand{\gen}[1]	{
	\ensuremath{
		\left\langle \, #1 \, \right\rangle
	}
}
\newcommand{\set}[1]	{
	\ensuremath{
		\left\{ #1 \right\}
	}
}

\begin{document}
\title{Normalisers of primitive permutation groups in quasipolynomial time}
\author{Colva M.  Roney-Dougal and Sergio Siccha}

\maketitle

\begin{abstract}
We show that given generators for subgroups $G$ and $H$ 
of $\Sn$, if $G$ is primitive 
then generators for $\Norm_H(G)$ may be computed in quasipolynomial
time, namely 
$2^{O(\log^3 n)}$. The previous best known bound was simply
exponential. 
\end{abstract}

\section{Introduction and outline}\label{sec:intro}

The \emph{normaliser problem} has as input generating sets $X$ and $Y$
for subgroups $G$ and $H$ of $\Sn$, and asks one to return generators
for $\Norm_H(G)$. 
The asymptotically best known  algorithm is due to Wiebking
\cite{wiebking:normalisers-simply-exponential},
and runs in simply exponential time $2^{O(n)}$. In
practice, there are no efficient general algorithms to solve the
normaliser problem.

Any generating set $X$ for a
subgroup $G$ of $\Sn$ can be reduced to one of size at most $n$ in
time $O(|X|n^2 + n^5)$ (see Lemma~\ref{lem:complexity-compendium}(a)),
so we shall generally assume that our generating sets have size at
most $n$.   A problem about subgroups of $\Sym_n$, input via their generating
sets,  can 
be solved in \emph{polynomial time}
if there exists an algorithm to solve it in time bounded polynomially in $n$. 
It can be solved in \emph{quasipolynomial} time if there exists an
algorithm to solve it in time bounded by $2^{O(\log^c
  n)}$, for some absolute constant $c$.

Polynomial time solutions are known for many 
permutation group problems.
However, in addition to the normaliser problem, polynomial time algorithms for
 computing set stabilisers, centralisers and intersections
of  permutation groups have so far proven
elusive.
Amongst these problems, the normaliser problem occupies a special place.
Luks showed  in \cite{luks:polytime} that the other problems listed above
are polynomial time equivalent: we 
call their polynomial time equivalence class the
\emph{Luks class}.
Luks also proved in \cite{luks:polytime} that 
the set-stabiliser problem is
polynomial time reducible to the normaliser problem, 
but no reduction in the other direction is known.

Although the normaliser problem appears to be harder than the problems
in the Luks class, they are not without
similarities.
Luks showed that 
the set-stabiliser problem is solvable in polynomial time if
one restricts the composition factors of the group in question.
This yielded a polynomial time algorithm for testing isomorphism of
graphs with bounded degree \cite{luks:isomorphism-graphs-bounded-valence}.
Almost 30 years later, Luks and Miyazaki were able to show that the same holds
for the normaliser problem: if one restricts the composition factors of
$H$ then one can compute $\Norm_H(G)$ in polynomial time
\cite{luks-miyazaki:normalisers-ptime}.

The majority of our paper consists of the proof of the following
theorem, concerning the solution of the normaliser problem in $\Sn$.
\begin{thm}\label{thm:normsym-qpoly}
    Let $G = \gen X \leq \Sn$ be a primitive permutation group. 
    Then $\Norm_\Sn(G)$ can be computed in time
    $|X| \cdot 2 ^ {O( \log ^ 3 n )}$.
\end{thm}
\begin{proof}
This will follow from Theorem~\ref{thm:normsym-small-qpoly} if $G$ is not
large  (see
Definition~\ref{def:large}), 
and from Corollary~\ref{cor:normsym-pa-qpoly} if $G$ is large.
\end{proof}

The \emph{string isomorphism} problem
(as stated in \cite[Definition 3.2]{babai:graphiso-quasipolynomial}) asks one
to find the elements $\sigma$ in a group $G \leq \Sn$ that map a function $f: \{1,
\ldots, n\} \rightarrow \Sigma$ to a function $g: \{1, \ldots, n\}
\rightarrow \Sigma$, where the action is
 $(i)f^{\sigma} = (i^{\sigma^{-1}})f$. In a dramatic
breakthrough, Babai proved  (see
\cite{babai:graphiso-quasipolynomial,
  babai:NEWgraphiso-quasipolynomial}) that the string isomorphism
problem
can be solved in quasipolynomial time. Babai also gave a polynomial
time reduction from the  graph isomorphism problem to the string
isomorphism problem, and hence showed that the graph isomorphism
problem can also be solved in
quasipolynomial time. 
The set
stabiliser problem is  a special case of the string isomorphism
problem, 
where $f = g$ and
 $|\Sigma| = 2$, and so it also follows  that all problems in the Luks class
 can be solved in quasipolynomial time. 

Helfgott went on to prove
a more precise bound in \cite{helfgott}, namely 
that the string isomorphism problem can be solved
in time
$2^{O(\log^3 n)}$.
Hence,  given subgroups $G$ and
$H$ of $\Sn$, 
the intersection
$G \intersection H$
can be computed in time $2^{O(\log^3 n)}$. 
Our main result now follows immediately:

\begin{thm}\label{thm:norm-qpoly}
    Let $G = \gen X \leq \Sn$ and $H = \gen Y \leq \Sn$ be permutation
    groups, with $G$ primitive.
    Then $\Norm_H(G)$ can be computed in time $(\abs X + \abs Y) 2^{O(\log^3 n)}$.
\end{thm}

A \emph{base} of a permutation group $G \leq \Sn$ is a sequence $B$ of
elements of $\{1, \ldots, n\}$, such that the pointwise stabiliser in
$G$ of
$B$ is trivial. We write $b(G)$ for the size of the smallest base for
$G$.  A 
technical tool in our proof, which may be of independent interest,
 is the following easy corollary of
beautiful recent work \cite{duyan-halasi-maroti, liebeck-halasi-maroti} 
of Duyan, Halasi, Liebeck and Mar\'oti. This bound is noted in
\cite[Proof of Corollary 1.3]{liebeck-halasi-maroti}, but we state and
prove it
in slightly more generality. 

\begin{cor}\label{cor:base_size}
Let $G$ be primitive and not large. Then 
$$b(G) \leq 2 \lfloor \log n
\rfloor + 26.$$
\end{cor}

\noindent (Throughout the paper, logarithms are to base $2$, unless specified otherwise.)

 In future work, we plan to consider the normaliser
problem for imprimitive groups, but the required techniques will be
quite different.   The proof of Theorem~\ref{thm:normsym-qpoly} relies on the fact that
primitive groups have generating sets of size $\mathrm{max}\{2, \log
n\}$ (see \cite{holt-RD}),
and either have a small base, as in Corollary~\ref{cor:base_size}, or
have a very precisely specified structure: in general, imprimitive groups have no
logarithmic bound on generating sets or base size. 

The paper is structured as follows.
In Section~\ref{sec:small-base-groups-and-giants}
we divide the primitive groups into small, large and almost simple
groups, state the results from which Corollary~\ref{cor:base_size}
will follow, and  collect various well known
facts about the complexity of permutation group algorithms. 
In Section~\ref{sec:a-backtrack-search-for-small-base-groups}
we present an algorithm to solve the normaliser problem in $\Sn$ when $G$ is
not large. 
In Section~\ref{sec:recognising_large_groups} we first show that 
the normaliser in $\Sn$ of the socle of any large primitive 
group of type PA can be computed in
polynomial time. 
We then give a ``moderately'' quasipolynomial time algorithm (time 
$2^{O(\log n \log \log n)}$)  to solve the normaliser
problem in $\Sn$ when $G$ is large and of type PA, and 
a polynomial time method when $G$ is almost simple.

\section{Small and large primitive groups}
\label{sec:small-base-groups-and-giants}

In this section we first divide the primitive groups into three
(non-disjoint) families: small groups, large groups and almost simple
groups. We then present some standard complexity results.

\subsection{Small and large primitive groups}

We write $\sym(\Omega)$ to denote the symmetric group on an arbitrary (finite)
set $\Omega$, and reserve $\Sn$ for the symmetric group on $\{1,
\ldots, n\}$. 

\begin{defn}\label{def:perm_isom}
Let $H \leq \sym(\Omega)$ and $K \leq \sym(\Gamma)$. A
\emph{permutation isomorphism} from $H$ to $K$ is a pair $(f, \phi)$,
where $f: \Omega \rightarrow \Gamma$ is a bijection, $\phi: H
\rightarrow K$ is an isomorphism, and $(\alpha^h)f = (\alpha f)^{h \phi}$ for all $\alpha \in \Omega$ and
$h \in H$.
\end{defn}

\begin{defn}\label{def:large}
Let $G \leq \Sn$ be primitive. Then $G$ is \emph{large} if there exist
natural numbers $m$, $k$ and $\ell$ such that  $n = \binom{m}{k} ^ \ell$, and 
    $G$ is permutation isomorphic to a group $\myhat G$ with
    $\mathrm{A}_m ^ \ell \leq \myhat G \leq \mathrm{S}_m \wr \mathrm{S}_\ell$,
    where the action of $\mathrm{A}_m$ and $\mathrm{S}_m$ is on
    $k$-element subsets of $\{1, \ldots, m\}$ and $\myhat G$ is in
    product action. (For a detailed description of product action, see
    \cite[Section 4.5]{dixon-mortimer:permutation-groups}.) 
\end{defn}


\begin{thm}[{\cite[Theorem 1.1]{maroti:orders-primitive-groups}}]\label{thm:maroti}
Let $G \leq \Sn$ be primitive. Then one of the following holds:
\begin{enumerate}
    \item
    $G$ is large;
    \item
    $G = \mathrm{M}_{11}$, $\mathrm{M}_{12}$, $\mathrm{M}_{23}$ or
    $\mathrm{M}_{24}$ with their $4$-transitive actions; or
    \item
    $|G| < n^{1 + \lfloor \log n \rfloor}$.
\end{enumerate}
\end{thm}

We shall say that $G$ is \emph{small} if $|G| < n^{1 + \lfloor \log n
  \rfloor}$. Thus each primitive group $G$ is small, large,
or almost
simple.

We shall combine Theorem~\ref{thm:maroti} with the following recent result, which
builds on Duyan, Halasi and Mar\'oti's proof
\cite{duyan-halasi-maroti} of Pyber's base size conjecture.

\begin{thm}[\cite{liebeck-halasi-maroti}]\label{thm:liebeck-halasi-maroti}
Let $G \leq \Sn$ be primitive. Then
$b(G) \leq \frac{2 \log|G|}{\log n} + 24.$
\end{thm}

\noindent {\sc Proof of Corollary~\ref{cor:base_size}:} Since $G$ is
primitive and not large, $G$ falls in Case (b) or (c) of
Theorem~\ref{thm:maroti}. If $G$ is in Case (b) then $b(G) \leq 7$,
and the result is immediate. If $G$ is in Case (c) then we use
Theorem~\ref{thm:liebeck-halasi-maroti} to see that
$$b(G) \leq \frac{2 \log |G|}{\log n} + 24 \leq \frac{2(1 + \lfloor \log n
\rfloor)\log n}{\log n}  + 24 \leq 2  \lfloor \log n \rfloor + 26$$ as required. 
\hfill $\Box$

\subsection{Complexity preliminaries}\label{subsec:complexity}

The following results are classical: see, for example 
\cite{furst-hopcroft-luks:basesgs-ptime, sims-1970}. Recall from
Section~\ref{sec:intro} that if we say that a permutation group
algorithm runs in \emph{polynomial time}, we mean polynomial in the
degree $n$. 

\begin{lemma}\label{lem:complexity-compendium}
Let $G = \langle X \rangle \leq \Sn$. 
\begin{enumerate}
\item In time $O(|X|n^2 +  n^5)$ we can replace $X$ by a generating set for
  $G$ of  size at most $n$. 
\item Given a generating set for $G$ of size at most $n$,
in polynomial time we can: compute the stabiliser in $G$ of any given
point; compute a base and strong generating
set for $G$; show that $G$ is primitive; compute
 $|G|$; find (at most $n$) generators for the socle $\soc G$; and compute
  the composition factors of $G$.
\item Given a base $B = (\beta_1, \ldots, \beta_b)$ and a strong generating
  set for $G$, 
in time $O(n^2)$ we can
test whether for a tuple $D \in \{1, \ldots, n\}^ b$
there exists $y \in G$ such that $B^y = D$, 
and if there exists such a $y \in G$ then determine $y$.
\end{enumerate}
\end{lemma}

%
    Let $H = \gen{a_1, \ldots, a_k} \leq \mathrm S_m$ and $K = \gen{b_1, \ldots, b_\ell} \leq \mathrm S_n$.
    We encode a homomorphism
    $\varphi \from H \to K$
    by a list
    $
        [ a_1, \ldots, a_k,
        b_1, \ldots, b_\ell,
        (a_1)\varphi, \ldots, (a_k)\varphi ].
    $
    We say that $\varphi$ is 
    \emph{given by generator images}. 
%


The following lemma is standard, and follows from 
Lemma~\ref{lem:complexity-compendium}. 

\begin{lemma}\label{lem:comp-with-homs-poly}
    Let $G_i = \gen{X_i} \leq \Sym_{n_i}$,  for
    $i = 1, 2, 3$.
    Furthermore let $\varphi \from G_1 \to G_2$
    and $\psi \from G_2 \to G_3$
    be homomorphisms given by generator images.
    Then in time polynomial in $|X_1|n_1 + |X_2|n_2 + |X_3|n_3$ we can:
     evaluate $\varphi$;
       compute $\varphi \psi$;
         compute  $\varphi^{-1}$, if $\varphi$ is an isomorphism;
        compute the restriction of $\varphi$ to $H_1$,  for any
        subgroup $H_1$ of $G_1$ with given generators. 
\end{lemma}

We shall also need to be able to compute a permutation isomorphism when given a
suitable group isomorphism. 
The following lemma is well known: see for example
\cite[Lemma 1.6B]{dixon-mortimer:permutation-groups} for a statement
equivalent to Part (a), and \cite[Lemma
3.5]{luks-miyazaki:normalisers-ptime} for a somewhat more general
statement than Part (b). 

\begin{lemma}\label{lem:permiso-from-groupiso-poly}
    Let $H = \gen A \leq \sym(\Omega)$ and $K = \gen B \leq
    \sym(\Gamma)$ both be transitive, 
    and let $\varphi \from H \to K$ be an isomorphism.
    \begin{enumerate}
    \item There exists a bijection $f$ such that $(f, \varphi)$ is a permutation
    isomorphism from $H$ to $K$ if and only if
    $\varphi$ maps each point stabiliser of $H$ to a point stabiliser of $K$.

    \item If such an $f$ exists, and $\varphi$ is given by  generator
      images, 
    then $f$ can be computed in time polynomial in $(\abs A + \abs B) |\Omega|$.
    \end{enumerate}
\end{lemma}

\section{The normaliser problem for small groups}
\label{sec:a-backtrack-search-for-small-base-groups}

In this section, we
 present an algorithm to compute the normaliser in $\Sn$ of a transitive
group of degree $n$,  in time exponential in $\log n$ and the sizes of a
given base and (not necessarily strong) generating set.
For primitive groups which are not large we show that this yields a
quasipolynomial time algorithm.

First we show that we can decide in polynomial time 
whether a primitive group $G$ is small.

\begin{lemma}\label{lem:rec-large-groups}
    Let $G = \gen X \leq \Sn$ be primitive, with $|X| \leq n$. 
    We can decide in polynomial time whether $G$ is small,
    and whether $G$ is almost simple. 
\end{lemma}
\begin{proof}
By Lemma~\ref{lem:complexity-compendium}(b), we
can verify that $G$ is primitive and
 compute the order of $G$ in polynomial time, and
hence determine whether $G$ is small. 
By the same lemma, we can also compute the composition factors of $\soc G$ in polynomial
time, and hence determine whether $\soc G$ is non-abelian
simple. 
\end{proof}

Next we show that
for primitive groups that are not large
we can find generating sets and bases of size $O(\log n)$
in quasipolynomial time.

\begin{lemma}\label{lem:find_gen_set}
Let $G = \gen X \leq \Sn $ be primitive and  not large, with $\abs X \leq
n$. 
\begin{enumerate}
\item 
We can compute a generating set of size $t \leq \max\{\log n, 2\}$ for
$G$ in time
$2 ^ {O( \log ^ 3 n )}$.
\item  We can compute a base $B$ of size $b \leq 2 \lfloor \log n \rfloor
+ 26$ for $G$
in time
$2 ^ {O( \log ^ 2 n )}$.
\end{enumerate}
\end{lemma}

\begin{proof}
(a) By  \cite{holt-RD} the group $G$ has such a generating set. 
The number of $t$-tuples of elements of $G$
is at most $|G|^t$. If $G$ is a Mathieu group then $|G|^t$ is bounded by
a constant, and otherwise by Theorem~\ref{thm:maroti} 
$$|G|^t < n^{t (1 + \lfloor \log n \rfloor)} <2^{2 t \log ^ 2 n} \leq  2^{2 \log ^ 3
  n}.$$ By Lemma~\ref{lem:complexity-compendium}(b), 
each such $t$-tuple can be tested for whether it generates
$G$ in polynomial time.
 
\medskip

\noindent (b)  By Corollary~\ref{cor:base_size}, such a $B$ exists.
There are  fewer than
$  n ^ {2 \lfloor \log n \rfloor + 26}   \leq  2 ^ {2\log ^ 2 n + 26 \log n}$
candidate $b$-tuples $B$ of elements from $\{1, \ldots, n\}$ to test. 
We can check, in polynomial time by
Lemma~\ref{lem:complexity-compendium}(b), whether $B$ is a
base.
\end{proof}

Our next result applies to all transitive groups, not just primitive
ones. 

\begin{thm}\label{thm:normsym-backtrack}
Let $G \leq \Sn$ be transitive. Assume that a generating set $X =
\{x_1, \ldots, x_t\}$ 
and  a base $B = (\beta_1, \ldots,
\beta_b)$  for $G$ are known. 
Then $\Norm_{\Sn}(G)$ can be computed 
in time $2^{O(tb\log n)}$.
\end{thm}

\begin{proof}
Let $\Omega = \{1, \ldots, n\}$. 
We iterate over all possible choices of $A = 
(\alpha_1, \ldots, \alpha_b) \in \Omega^b$,
and $D= (\delta_{ij} \ : \ 1 \leq i \leq t, 1 \leq j \leq
b) \in \Omega^{bt}$. For each such $A$ and $D$ we proceed as follows. 

Let $\gamma_{ij} = \alpha_j^{x_i}$, and let $\overline{\sigma}: \alpha_j \mapsto
\beta_j$, $\gamma_{ij} \mapsto \delta_{ij}$. We first test whether
$\overline{\sigma}$ is a well-defined bijection from its domain $\Delta \subseteq
\Omega$ to its image $\Gamma \subseteq \Omega$. To do so, we
check
whether the $\alpha_i$ are pairwise distinct, whether
 $\gamma_{{i_1}{j_1}} = \gamma_{{i_2}j_2}$ if and only if 
$\delta_{i_1 j_1} = \delta_{i_2 j_2}$, and whether $\gamma_{ij} =
\alpha_k$ if and only if $\delta_{ij} = \beta_k$. 
If the answer is ever ``no'', then we  move on to the next choice of $A$ and
$D$. For a fixed $A$ and $D$, this step requires 
time $O(n^2t^2)$. 

If $\sigma \in \Sn$ satisfies $\sigma|_{\Delta} = \overline{\sigma}$, then 
$x_i^{\sigma}: \beta_j \mapsto \delta_{ij}$ for $1 \leq i \leq t$  and $1 \leq j \leq b$. 
Since $B$ is a base, for $1 \leq i \leq t$  there is at most one $y_i \in
G$ such that  $B^{y_i} = (\delta_{i1}, \ldots, \delta_{ib})$. 
The existence of such $y_i$, and their determination, can be calculated in
time $O(n ^ 2t)$ by Lemma~\ref{lem:complexity-compendium}(c).  If for
some $i$ no corresponding $y_i$ exists, then we move on to the next choice of
$A$ and $D$. 

We now show how to determine the (unique) $\sigma \in \Sn$ such that
$x_i^\sigma = y_i$ for $1 \leq i \leq t$, or show that no such
$\sigma$ exists. 
The identity $\mu ^ {x_i
  \sigma} = \mu ^ {\sigma y_i}$ must hold
for all $\mu \in \Delta$ and all $x_i \in X$.
Since $G$ is transitive,
if $\Delta \subsetneq \Omega$ we can find
$x_j \in X$ and $\mu \in \Delta$
with $\mu ^ {x_j} \not \in \Delta$.
Then 
$(\mu ^ {x_j})^\sigma = \mu^{\sigma y_j} =  \mu ^ {\overline \sigma y_j} 
\not\in \Gamma.$
Hence we can define and check the image of one more
point under $\sigma$
by examining at most $|\Delta||X| = O(nt)$ images of points under
permutations until we find a $\mu$ and $x_j$ as specified. This must
be carried out $O(n)$ times to fully specify $\sigma$, so the total time  is
$O(n^2t)$.

Hence, given $A$ and
$D$,  in time $O(n^2t^2)$
we can either
determine a corresponding element $\sigma \in \Norm_{\Sn}(G)$, or show that
no such element exists.
There are $n^{b + tb} = 2^{(t+1)b \log n}$ such
sequences of elements $A$ and $D$ to test.
Thus we can compute
$\Norm_{\Sn}(G)$ in time
$O(n ^ 2 t^2 \cdot 2^{(t+1)b \log n})
\subseteq 2 ^ {O(t \, b \log n)}$.
\end{proof}

Finally we present our main result for primitive groups that are not large.

\begin{thm}\label{thm:normsym-small-qpoly}
    Let $G = \gen X \leq \Sn$ be primitive and not large.
    Then we can compute $\Norm_\Sn(G)$ in time
    $|X| \cdot 2 ^ {O(\log ^ 3 n)}$.
\end{thm}

\begin{proof}
By Lemma~\ref{lem:complexity-compendium}, in $O(|X|n^5)$ we can
replace $X$ by a generating set of size at most $n$, and in
polynomial time we can compute a
base and strong generating set for $G$. By
Lemma~\ref{lem:rec-large-groups} we can deduce that $G$ is small or
almost simple in polynomial time, and hence whether $G$ is not large.
The result now follows from
Lemma~\ref{lem:find_gen_set} and
Theorem~\ref{thm:normsym-backtrack}.
\end{proof}

\section{The normaliser problem for large groups}
\label{sec:recognising_large_groups}

Recall Definition~\ref{def:large} of a large primitive group.
 In this section we show that if $G$ is  large then 
we can construct $\Norm_{\Sn}(G)$ in polynomial time.

\begin{defn}\label{def:PA}
    Let $G \leq \Sn$ be primitive.
    The group  $G$ is of \emph{type  PA} if there exist an
    $\ell \geq 2$, an almost simple primitive group $A \leq
    \sym(\Delta)$ with socle $T$, 
    and a group
    $\myhat G \leq \sym(\Delta ^ \ell)$
    permutation isomorphic to $G$
    with
    \[
        \hspace{3em}
        \soc G \cong T ^ \ell 
        \leq
        \myhat G
        \leq A \wr \mathrm S_\ell,
    \]
    in product action on $\Delta ^ \ell$.
\end{defn}

From Definition~\ref{def:large}, 
 it is immediate that a large primitive group  is
either almost simple or of type PA.

\subsection{Constructing the normaliser of the socle}

In this subsection, we construct $\Norm_{\Sn}(\soc G)$
when $G$ is large and of type PA.  
We will require the following well known property: 
see \cite[Lemma 4.5A]{dixon-mortimer:permutation-groups}, for example.

\begin{lemma}\label{lem:properties-prim}
Let $G$ be primitive of type PA. With the notation of 
Definition~\ref{def:PA}, 
        $\Norm_{\sym(\Delta ^ \ell)} (\soc \myhat G) =  \Norm _
        {\sym(\Delta)} ( T ) \wr \mathrm S_\ell$, in product action. 
\end{lemma}

\begin{lemma}\label{lem:normsym-soc-pa-poly}
   Let $G  = \langle X \rangle \leq \Sn$ be a large primitive group
    of type PA, with $|X| \leq n$. Then
    we can compute a generating set of size four for
   $\Norm_{\Sn}(\soc G)$ in polynomial time.
\end{lemma}
\begin{proof}
Let $S = \soc G$. Since $G$ is large and of type PA, there exist an $\ell \geq 2$, $m$ and $k$ 
such that $S$ 
is permutation isomorphic to $\Alt_m ^ \ell$ acting component-wise on
the set of $\ell$-tuples of $k$-element subsets of $\{1, \ldots,
m\}$. Let $\Delta$ be the set of $k$-subsets of $\set{1, \ldots, m}$, and
write $\Alt_{k, m}$ and $\Sym_{k, m}$ for  the actions of $\Alt_m$ and
$\Sym_m$ on $\Delta$. Let $\Gamma = \Delta^{\ell}$, and let
$W= \Sym_{k, m} \wr \Sym_{\ell} \leq \sym(\Gamma)$, in product
action. Let $A:= \soc W = \Alt_{k, m}^{\ell} \leq \sym(\Gamma)$. 
Then $S$ is permutation isomorphic to $A$. We shall proceed by first
constructing the permutation isomorphism, and then using it to
construct $\Norm_{\Sn}(S)$ as the pre-image of
$\Norm_{\sym(\Gamma)}(A)$. It is immediate from Lemma~\ref{lem:properties-prim}
that $\Norm_{\sym(\Gamma)}(A) = W$.

We first compute a base and strong generating set for $G$, and find
generators for,  and the  composition factors of, $S$.
Hence we determine $\ell$ and $m$, and
calculate $k$ from $n = {m \choose k} ^ \ell$. This can all be done in polynomial time
by Lemma~\ref{lem:complexity-compendium}.

We shall now proceed in three steps.  
First we will compute an isomorphism $\iota$ from $S$ to 
$A$, and generators for $W = \Norm_{\sym(\Gamma)}(A)$. 
Secondly, we will use $\iota$ to construct a permutation isomorphism $(f,
\varphi)$ from $S$ to $A$. Finally, we shall use $f$ to construct
generators for $\Norm_{\Sn}(S)$. 

Let
$W_1 \leq \Sym_{ml}$ be the wreath product of $\Sym_m$ and
$\Sym_\ell$, in imprimitive action. We generate $W_1$ with a set
$Z$ of size four: two permutations
acting non-trivially only on the first block, to generate one copy
of $\Sym_m$, and two permutations generating the top group $\Sym_{\ell}$. 
Note that $m\ell \leq \abs \Delta \ell \leq n$, so $Z$ can be written down in
polynomial time. 
By \cite[Theorem 4.1]{babai-luks-seress:permutation-groups-nc} we can compute
in polynomial time an  isomorphism
$\iota_{1, 1} \from G \to K$
for some $K \leq W_1$. 
Notice that $\soc K = \soc W_1$.
We can write down the corresponding two generators for the action of $\Sym_{k,
  m}$ on $\Delta$, and lift their action 
on $\Delta$ to an action on
$\Gamma = \Delta ^ \ell$ by letting them act trivially on all but the first components of
$\Gamma$.
Then we write down two generators for the action of $\Sym_\ell$
on the components of $\Gamma$, corresponding to those in $Z$.
 Let $Y$ be this set of four permutations. 
Mapping $Z$ to $Y$ yields an  isomorphism by generator images
$\iota_{2, 1} \from W_1
\to
W = \Sym_{k,m} \wr \Sym_\ell \leq \sym(\Gamma)$.
Since $W$ acts on $\abs \Delta ^ \ell = n$ points we can 
compute and evaluate $\iota_{2, 1}$ in polynomial time, by 
Lemma~\ref{lem:comp-with-homs-poly}.
We let $\iota_1$ and
$\iota_2$ be the restrictions of $\iota_{1, 1}$ and $\iota_{2,
1}$ to the socles of the groups concerned:
we can compute
the socles and restrictions in polynomial time
by Lemma~\ref{lem:complexity-compendium}
and Lemma~\ref{lem:comp-with-homs-poly}, respectively.
Finally, we let $\iota = \iota_1 \iota_2$ be
the isomorphism from $S$ to $A$ given by the 
composition of these maps.

The groups $A$ and $S$ are 
permutation isomorphic: denote one such permutation isomorphism by 
\begin{equation}\label{eq:fixed_isom}
(p, \psi), \mbox{ with } \psi \from A \to S, \quad p \from
\Delta^{\ell} \to \{1,
\ldots, n\}, 
\end{equation} and notice that $\psi \iota \in \aut(A)$.

We shall now use $\iota$ to construct a 
permutation isomorphism $(f, \varphi)$ from $S$ to $A$. By
Lemma~\ref{lem:permiso-from-groupiso-poly}(b), it suffices to find a
suitable $\varphi$, as then $f$ can be constructed in polynomial
time. By Lemma~\ref{lem:permiso-from-groupiso-poly}(a), 
such an $f$ exists if and only if $\varphi$ maps the point stabilisers
of $S$ to the point stabilisers of $A$.

To determine $\varphi$, we first 
consider the homomorphism:
\[
    \Psi \from W \to \aut(A),~
    x \mapsto (g \mapsto x ^ {-1} \! g x).
\]
Recall that $\Norm_{\sym(\Gamma)}(A) = W$.
Thus $\mathrm{Im}(\Psi)$ consists of exactly those automorphisms
of $A = \Alt_{k,m} ^ \ell$ which are induced by permutations of $\Gamma$.

Assume first that $m = 5$ or $m \geq 7$, so that $\Psi$ is surjective.
In this case we shall define $\varphi := \iota$. To see that
$\varphi$ maps point stabilisers to point
stabilisers, let $\psi$ be as in \eqref{eq:fixed_isom}, so that
$\psi\iota = \psi \varphi \in \aut(A)$. 
Then there exists an element
$w \in W \leq \sym(\Gamma)$ with $w\Psi = \psi \varphi$. 
Hence $\psi \varphi$ permutes the point stabilisers of $A$, and so
  by  Lemma~\ref{lem:permiso-from-groupiso-poly}(a) the map 
$\varphi$ induces a permutation isomorphism.

Now let $m = 6$, so that  the image of $\Sym_{k,6}$ under
$\Psi_0 \from \Sym_{k,6} \to \aut(\Alt_{k,6})$ has index 2 in $\aut(\Alt_{k,6})$.
Correspondingly the index $[\aut (A) : W \Psi] = 2^\ell$.
In constant time we can compute generator images for 
an involutory 
automorphism $\tau \in \aut(\Alt_{k, 6}) \setminus (\Sym_{k,6}\Psi_0)$. 
Hence we can compute generator images for $\tau_1, \ldots, \tau_\ell
\in \aut(A)$,  where each $\tau_i$ induces $\tau$ on a distinct
direct factor of $A$,  and the identity on the other factors. Notice that
$L:= \langle \tau_1 \ldots, \tau_\ell \rangle \cong C_2^{\ell}$, so
that $|L| = 2^{\ell}  \leq
n$.

Let $(p, \psi)$ be  as in \eqref{eq:fixed_isom}, so that $\psi
\iota \in \aut(A)$. 
Then there exists a $\theta \in L$ such that
$\psi  \iota \theta \in W \Psi$.
In particular, 
 $\psi  \iota \theta$  maps point stabilisers to point stabilisers, and hence 
by  Lemma~\ref{lem:permiso-from-groupiso-poly}(a) 
$\iota \theta$ does so too. 
We can test in polynomial time whether the image of a point stabiliser
is a point stabiliser, so we can
 check each of the $O(n)$ elements of $L$ to find such a $\theta$ in
polynomial time. We then set $\varphi = \iota \theta$.

Finally, the bijection $f^{-1}$ induces an isomorphism
from $\sym(\Gamma)$ to $\Sn$ that maps
$W = \Sym_{k, l} \wr \Sym_\ell$
to the normaliser $\Norm_\Sn(S)$. Recall the set $Y$ of generators for
$W$. We can compute
$Y{f^{-1}}$ in polynomial time, and $\gen{Yf^{-1}} = \Norm_{\Sn}(S)$. 
\end{proof}

\subsection{Computing the normaliser in $\Sn$ of a large group}
\label{subsec:the_strategy_for_giants}

Many of our results in this subsection apply directly to all
primitive groups $G$ that are either almost simple or of type PA.
Our approach yields a
``very moderately'' quasipolynomial time algorithm: $O(2^{\log n \log
  \log n})$. 

Throughout this subsection, let $G \leq \Sn$, and let $M
= \Norm_{\Sn}(\soc G)$. We first relate the complexity of
computing $\Norm_\Sn(G)$ to the index $[M : G]$.  

\begin{lemma}\label{lem:normsym-if-norm-socle}
    Let $G = \gen X \leq \Sn$ be primitive, with $|X| \leq n$. Assume
    that a base and strong generating set for $G$ are known. 
    Furthermore, let a generating set $Y$ for 
    $M$ be given, with $|Y| \leq n$. 
    Then we can compute the normaliser $\Norm_\Sn(G)$ in time
    $O(n^3 [M:G]^2)$.
\end{lemma}

\begin{proof}
First notice that  $\soc G$ is normal in $\Norm_\Sn(G)$, and so
 $\Norm_\Sn(G)$ is a subgroup of $M$.
   Since  $G \leq M$, to compute $\Norm_{\Sn}(G)$  it suffices to test
    only the representatives of the
    right cosets $G \backslash M$ of $G$ in $M$.

    We shall compute the 
    representatives of $G \backslash  M$
    by a standard orbit algorithm.
    To determine equality of cosets $Ga$ and $Gb$ we test whether
    $ab^{-1} \in G$, in time $O(n ^ 2)$ by
    Lemma~\ref{lem:complexity-compendium}(c).
    The cost of the computation of these coset representatives is then the product of
    the number of generators $\abs Y$,
    the size of the orbit $[M : G]$,
    and the cost
    $[M : G] \cdot n ^ 2$
    for testing whether a representative of a newly computed coset
    $(Ga)y$ for $y \in Y$
    is contained in one of the up to $[M : G]$ already computed cosets.
    In total this takes
    $O(\abs Y \cdot n ^ 2 \cdot [M : G] \cdot [M : G])  = O(n^3 [M:G]^2)$ time.

    Next, for $h$  a coset representative, we can test whether $x^h \in
    G$ for all $x \in X$ in time $O(n ^ 3)$.
    We can therefore test all such representatives $h$ in time
    $O(n ^ 3 \cdot [M \nobreak : \nobreak G])$.
    The result follows. 
\end{proof}

Luks and Miyazaki showed  in \cite[Corollary
3.24]{luks-miyazaki:normalisers-ptime} 
that normalisers in $\Sn$ of non-abelian simple groups can be
computed in polynomial time. We very slightly extend their argument, 
to almost simple groups: the following result does not require transitivity.

\begin{lemma}\label{lem:normsym-as-type-poly}
    Let $G = \gen X \leq \Sn$ be almost simple, with
    $|X| \leq n$. 
    Then we can compute $\Norm_{\Sn}(G)$ in polynomial time.
\end{lemma}
\begin{proof}
By Lemma~\ref{lem:complexity-compendium}(b), we can compute a base and
strong generating set for $G$, and generators for
$T:= \soc G$,  in polynomial time. 

By \cite[Lemma 3.3]{luks-miyazaki:normalisers-ptime} we can compute
$C:= C_{\Sn}(G)$, and then arguing as in the proof of 
\cite[Corollary 3.24]{luks-miyazaki:normalisers-ptime} we can
determine in polynomial time which elements of $\aut T$ are induced by
$N_{\Sn}(T)$, and construct an almost simple group $A$ with socle $T$ which
induces all of these automorphisms. 
Let $M = \langle C, A \rangle$. 
Then $M = \Norm_{\Sn}(T)$.
By \cite[Lemma 7.7]{guralnick-maroti-pyber:normalisers-primtive-groups},  there
exists an absolute constant $\kappa$ such that
$\abs{\out T} \leq \kappa\sqrt n$.
Hence $[M : G] \leq \kappa \sqrt n$. The result is now immediate
from Lemma~\ref{lem:normsym-if-norm-socle}.
\end{proof}

Next we consider the normalisers of the socles of groups of type PA. 

\begin{prop}\label{prop:pa-small-index-socle-in-norm}
    There exists a constant $c$ such that for all 
    groups $G \leq \Sn$ of type PA
    \[
        [ M : \soc G ]
        \leq
        2 ^ {c \log n \log \log n}.
    \]
\end{prop}
\begin{proof}
Let $T \leq \sym \Delta$ be a non-abelian simple group such that
$S:= \soc G$ is permutation isomorphic to $T ^ \ell \leq \sym \Delta ^
\ell$. 
Then
$M  \cong \Norm_{\sym \Delta}(T) \wr \mathrm{S}_\ell$, by 
Lemma~\ref{lem:properties-prim}.
Now 
$\ell = \log_{\abs \Delta} n \leq \log n$
and
by
\cite{guralnick-maroti-pyber:normalisers-primtive-groups}
there exists a constant $\kappa$ such that 
$\abs{\out T} \leq \kappa \sqrt{ \abs \Delta }
= \kappa n^{1/(2\ell)}$. Let $\lambda = \log \kappa$. 
Then
\begin{align*}
    [ M : S ]
    &\leq
    \abs{ \out T } ^ \ell \abs{\mathrm S_\ell} \leq
    \left(  \kappa n^{1/(2\ell)} \right) ^ \ell
    \cdot \ell ^ \ell \leq
    2^{\lambda \log n} \cdot n^{1/2} 
    \cdot \left( \log n \right) ^ {\log n} \\
& \leq
    2 ^ { \lambda \log n + (1/2)\log n +  \log n \log \log n} = 2^{c \log n \log \log n},
\end{align*}
for a suitable choice of $c$. 
\end{proof}

\begin{prop}\label{prop:norm_large_PA}
Let $G = \gen X  \leq \Sn$ be primitive of type PA, with
$|X| \leq n$. If a generating set $Y$ of size at most $n$ is known for
$M$, 
then we can compute the normaliser $\Norm_\Sn(G)$ in time $2^{O(\log n \log \log n)}$. 
\end{prop}

\begin{proof}
By Lemma~\ref{lem:complexity-compendium}(b) we can compute a base
and strong generating set for $G$
 in polynomial time. 
Given this data and $Y$, by
Lemma~\ref{lem:normsym-if-norm-socle} we can compute $\Norm_{\Sn}(G)$ in
time $O(n^3[M:G]^2)$.
 By Proposition~\ref{prop:pa-small-index-socle-in-norm} the
 index $[M:G] \leq 2^{c \log n \log \log n}$ for some constant $c$.  
 The
result follows.
\end{proof}

The following corollary completes the proof of
Theorem~\ref{thm:normsym-qpoly}. 

\begin{cor}\label{cor:normsym-pa-qpoly}
    Let $G = \gen X \leq \Sn$ be a large primitive group.
    Then the  normaliser $\Norm_\Sn(G)$ can be computed in time
    $\abs X \cdot 2 ^ {O(\log n \log \log n)}$.
\end{cor}

\begin{proof}
By Lemma~\ref{lem:complexity-compendium} we can replace $X$ by a set
of at most $n$ generators in $O(|X|n^2 + n^5)$, and we can determine
whether $G$ is almost simple in polynomial time.  If $G$ is almost
simple, then we can compute $\Norm_{\Sn}(G)$ in polynomial time, by
Lemma~\ref{lem:normsym-as-type-poly}. Otherwise, $G$ is of type PA.  By
Lemma~\ref{lem:normsym-soc-pa-poly} we can construct a generating set
of size four for $M:= \Norm_{\Sn}(\soc G)$ in polynomial time. 
Hence by Proposition~\ref{prop:norm_large_PA}  we can compute
$\Norm_{\Sn}(G)$ in time $2^{O(\log n \log \log n)}$. 
\end{proof}

\paragraph{Acknowledgements} The second author is 
supported by DFG research training group ``Experimental and
  Constructive Algebra''. The first author would like to thank the
  Isaac Newton Institute for Mathematical Sciences for support and
  hospitality during the programme ``Groups, Representations and
  Applications: New perspectives'', when work on this paper was undertaken. This work was supported by:
EPSRC grant number EP/R014604/1.  We thank Professor Babai for drawing our
  attention to the relationship between the string isomorphism problem
  and the Luks class.

\bibliographystyle{plain}
\bibliography{./references}


\end{document}